\documentclass[12pt,oneside]{amsart}
\usepackage{latexsym,amssymb,amsmath,graphicx,hyperref,enumitem}
\usepackage{color}

\numberwithin{equation}{section}
\newtheorem{theorem}{Theorem}[section]
\newtheorem{lemma}[theorem]{Lemma}

\newtheorem{corollary}[theorem]{Corollary}
\newtheorem{conjecture}[theorem]{Conjecture}
\newtheorem{question}[theorem]{Question}

\theoremstyle{definition}

\newtheorem{remark}[theorem]{Remark}
\newtheorem{example}[theorem]{Example}

\newcommand{\popo}{\mathbb{P}^1 \times \mathbb{P}^1}
\renewcommand{\P}{\mathbb{P}}
\newcommand{\ffi}{\varphi}
\DeclareMathOperator{\Ass}{Ass}
\DeclareMathOperator{\Sym}{Sym}
\DeclareMathOperator{\rk}{rank}
\DeclareMathOperator{\codim}{codim}
\DeclareMathOperator{\depth}{depth}
\DeclareMathOperator{\h}{ht}

\newcommand{\sat}{{\rm{sat}}}
\newcommand{\fm}{\mathfrak{m}}
\newcommand{\fp}{\mathfrak{p}}

\begin{document}

\title[Powers of codimension 2
Cohen-Macaulay ideals]{Symbolic powers of codimension two Cohen-Macaulay
ideals}
\thanks{Last updated: (Final Version - May 4, 2020)}

\author{Susan Cooper}
\address{Department of Mathematics\\
University of Manitoba\\
420 Machray Hall, 186 Dysart Road\\
 Winnipeg, MB R3T 2N2 Canada}
\email{Susan.Cooper@umanitoba.ca}

\author{Giuliana Fatabbi}
\address{Dipartimento di Matematica e Informatica\\
Via Vanvitelli, 1 \\
06123, Perugia, Italy}
\email{giuliana.fatabbi@unipg.it}

\author{Elena Guardo}
\address{Dipartimento di Matematica e Informatica\\
Viale A. Doria, 6 \\
95100 - Catania, Italy} \email{guardo@dmi.unict.it}

\author{Anna Lorenzini}
\address{Dipartimento di Matematica e Informatica\\
Via Vanvitelli, 1 \\
06123, Perugia, Italy}
\email{annalor@dmi.unipg.it}

\author{Juan Migliore}
\address{Department of Mathematics \\
University of Notre Dame \\
Notre Dame, IN 46556, USA}
\email{migliore.1@nd.edu}

\author{Uwe Nagel}
\address{Department of Mathematics\\
University of Kentucky\\
715 Patterson Office Tower\\
Lexington, KY 40506-0027, USA}
\email{uwe.nagel@uky.edu}

\author{Alexandra Seceleanu}
\address{Department of Mathematics\\
University of Nebraska\\
Lincoln, NE 68588-0130, USA}
\email{aseceleanu@unl.edu}

\author{Justyna Szpond}
\address{Department of Mathematics\\
Pedagogical University of Cracow\\
Podchor\c a\.zych 2\\
PL-30-084 Krakow, Poland}
\email{szpond@gmail.com}

\author{Adam Van Tuyl}
\address{Department of Mathematics and Statistics,
McMaster University, Hamilton, ON, L8S 4L8, Canada}
\email{vantuyl@math.mcmaster.ca}


\begin{abstract}
Let $I_X$ be the saturated homogeneous ideal defining
a codimension two
arithmetically Cohen-Macaulay scheme $X \subseteq \mathbb{P}^n$,
and let $I_X^{(m)}$ denote its $m$-th symbolic power.
We are interested in when $I_X^{(m)} = I_X^m$. 
We survey what is known
about this problem when $X$ is locally a complete intersection,
and in particular, we review the classification of
when $I_X^{(m)} = I_X^m$ for all $m \geq 1$.  We then discuss
how one might weaken these hypotheses, but still obtain
equality between the symbolic and ordinary powers.   Finally,
we show that this classification allows one to: (1) simplify known
results about symbolic powers of ideals of points in $\popo$; 
(2) verify a conjecture of Guardo, Harbourne, and Van Tuyl,  and
(3) provide additional evidence to a conjecture of R\"omer.
\end{abstract}

\maketitle

%%%%%%%%%%%%%%%%%%%%%%%%%%%%%%%%%%%%%%%%%%%%%%%%%%%%%%%%%%%%%%%%%%%%%%%%%%%%%%%
\section{Introduction}\label{sec:intro}

Throughout this paper, $R = k[x_0,\ldots,x_n]$ where $k$ is an
algebraically closed field of characteristic zero.  For any
non-zero homogeneous ideal $I \subseteq R$,
the {\it $m$-th symbolic power} of $I$, denoted $I^{(m)}$, is
the ideal
\[
I^{(m)}= \bigcap_{\fp \in {\rm Ass}(I)}(I^mR_\fp \cap R),
\]
where $R_\fp$ denotes the localization of $R$ at the prime ideal
$\fp$, and ${\rm Ass}(I)$ is the set of associated primes of $I$.
In general, $I^m \subseteq I^{(m)}$, but the reverse containment
may fail.   Fixing $m$, the {\it ideal containment problem} asks
for the smallest integer $r$ such that $I^{(r)} \subseteq I^m$.
The papers of  Ein, Lazarsfeld, and Smith \cite{ELS}, Hochster
and Huneke \cite{HH}, and Bocci and Harbourne \cite{BH1,BH} are
among the first papers to systematically study this problem.
Recent work includes \cite{BCH,DHNSST,HaHu13,NS,S}; see also
the surveys \cite{DDGHN,SS} and book \cite{CHHVT}.

A complementary problem, and one which we consider in
this paper, is to ask for conditions on $I$ that force $I^m = I^{(m)}$
for all $m \geq 1$.  This problem is equivalent to asking when $r = m$
in the ideal containment problem, the smallest value that $r$ could have.
A classical result
in this direction is a result of Zariski and Samuel
\cite[Lemma 5, Appendix 6]{ZS}
that states $I^m = I^{(m)}$ for all $m \geq 1$ if $I$ is generated by a regular
sequence, or equivalently, a complete intersection. Ideals that have the
property $I^m = I^{(m)}$ for all $m\geq 1$ are called normally torsion free
because  their  Rees algebra is normal. The normally torsion free squarefree
monomial ideals were classified by  Gitler, Valencia and Villarreal \cite{GVV}.
They showed that a squarefree monomial ideal is normally torsion free if and
only if the corresponding hypergraph satisfies the max-flow min-cut property.
 Simis, Vasconcelos and Villarreal \cite{SVV} and separately Sullivant \cite{Su}
showed that edge ideals of graphs are normally torsion free if and only if
the graph is bipartite. Furthermore, Olteanu \cite{O} characterizes normally
torsion free ideals that are lexsegment.
More recent work on the equality between symbolic and ordinary powers
includes Morey's paper \cite{Mo} on
a local version of this question, Guardo, Harbourne,
and Van Tuyl's paper \cite{GHVT2} which
identifies all the  ideals of general points in
$\popo$ that satisfy $I^m = I^{(m)}$ for all $m \geq 1$, and
Hosry, Kim, and Validashti's work \cite{HKV} which identifies
some families of prime ideals $P$ such that $P^m \neq P^{(m)}$.

One family of ideals for which a complete answer is known to the
question of when $I^m = I^{(m)}$ is the family of 
ideals defining a codimension two arithmetically
Cohen-Macaulay subscheme of $\mathbb{P}^n$ that is also
locally a complete intersection.  In particular:

\begin{theorem}
\label{thm:all powers equal}
Let $I = I_X$ be the saturated homogeneous ideal defining a subscheme $X \subset\P^n$
such that
\begin{itemize}
   \item $\codim(X)=2$;
   \item $X$ is arithmetically Cohen-Macaulay;
   \item $X$ is locally a complete intersection.
\end{itemize}
Then the following conditions are equivalent:
\begin{itemize}
   \item[(a)] $I^{(n)} = I^n$;
   \item[(b)] $I^{(m)}=I^m$ for all $m\geq 1$;
   \item[(c)] $I$ has at most $n$ minimal generators.
\end{itemize}
Furthermore, if $m<n$, then $I_X^{(m)} = I_X^m$
regardless of the number of generators.
\end{theorem}

\noindent
Theorem \ref{thm:all powers equal} is a graded version of work
of Ulrich \cite{U}  and Morey \cite{Mo} which considered local versions
of this problem.  Given the current interest in the containment
problem for homogeneous ideals, the purpose of Section 2 is to
provide a proof of Theorem \ref{thm:all powers equal} in the graded
case.  In fact, we give a slightly more general result
by also proving a similar statement for codimension three arithmetically
Gorenstein schemes.

We are additionally interested in the graded minimal free 
resolutions (i.e., Betti numbers) of ideals that meet the hypotheses of  Theorem \ref{thm:all powers equal}.
Using the graded strands of a certain Koszul complex, we 
show that under suitable technical hypotheses,
one can construct the graded minimal free resolutions of powers
of perfect homogeneous ideals of codimension two
(see Theorem \ref{thm:EN exact}).   As we describe in
Remarks \ref{altapproach} and \ref{HSVremark}, these resolutions
could also be constructed using \cite{HSV},  \cite{ABW} or \cite{Tc};
we present a more targeted proof of this known result.

Our new contributions are in the remaining sections.
In Section 3,  we discuss the relative importance of the different hypotheses
in Theorem~\ref{thm:all powers equal} by providing a menagerie of examples.
Among other things, we show  that the assumption
that $X$ be arithmetically Cohen-Macaulay is essential in codimension two.
If it is dropped,
then it might happen that all symbolic and ordinary powers
of $I$ are equal but $I$ has more than $n$ generators --
see Example \ref{ex:5 lines} and
Theorem \ref{thm: complete intersections codim4}.
In contrast, we provide some evidence that
the condition on the codimension can be
extended to codimension three in the ACM situation --
see Examples \ref{ex:P3P4} and \ref{ex:scrolls}.
All of these considerations lead us to ask Question \ref{Q1}  as an indication
of possible future directions for this investigation.
Finally, we examine the hypothesis that $X$ be locally a complete intersection and propose
Conjecture \ref{conj:not LCI}.   A collection of examples concerning ideals of fat points in
$\mathbb{P}^2$ (see Examples \ref{ex:lci 1}, \ref{ex:lci 2} and \ref{ex:lci 3}) which satisfy
$I^{(m)} = I^m$ is also provided.

In Section 4 we present an application of Theorem
\ref{thm:all powers equal} to points in $\P^1 \times \P^1$.  Specifically,
we verify the  following statement, which was conjectured in
\cite[Conjecture 4.1]{GHVT}. (See Corollary \ref{P1xP1 result} for
a slightly more precise statement.)

\begin{corollary}
Let $I = I_X$ be the saturated defining ideal of an arithmetically Cohen-Macaulay
set of points in $\P^1 \times \P^1$.  Then $I^{(3)} = I^3$ if and
only if $I$ is a complete intersection
or $I$ is an almost complete intersection
(i.e., it has exactly three minimal generators).
\end{corollary}

\noindent
In addition, we show how  Theorem
\ref{thm:all powers equal} significantly simplifies earlier
arguments of Guardo and Van Tuyl \cite{GVT2} and Guardo,
Harbourne, and Van Tuyl \cite{GHVT}.
In fact, the original motivation of this project
was to prove \cite[Conjecture 4.1]{GHVT}.   We were initially
able to verify this conjecture using
Peterson's \cite{P} results on quasi-complete intersections, which
first suggested the importance of being locally a complete intersection.  Generalizing
our specialized proof lead to the much stronger results of this paper.

In the final section, we use
the minimal graded free resolution given in
Theorem \ref{thm:EN exact} to verify that for small powers of codimension two
perfect ideals that are locally complete intersections,
a question of R\"omer \cite{R} has an affirmative answer.  In
the case of ACM sets of points in $\P^1 \times \P^1$, we also have
a new proof of a result of Guardo and Van Tuyl
\cite{GVT2}.
\smallskip

\noindent
{\bf Acknowledgments.}  This project was started at the
Mathematisches Forschungsinstitut Oberwolfach (MFO) as part of the
mini-workshop ``Ideals of Linear Subspaces, Their Symbolic Powers
and Waring Problems'' organized by C. Bocci, E. Carlini, E.
Guardo, and B. Harbourne.     All the authors thank the MFO for
providing a stimulating environment.  We also thank Brian Harbourne
and Tomasz Szemberg
for their feedback on early drafts of the paper.
Cooper
acknowledges support from the NDSU Advance FORWARD program sponsored by the
National Science Foundation, HRD-0811239, as well as financial support provided by NSERC.
Fatabbi and Lorenzini were partially supported by GNSAGA (INdAM).
Guardo acknowledges the
financial support provided by Prin 2011 and GNSAGA (INdAM).
Migliore and Nagel were partially supported
by  Simons Foundation grants  (\#309556 for Migliore, \#317096 and \#636513 for
Nagel).
Seceleanu received support from MFO's NSF grant
DMS-1049268, ``NSF Junior Oberwolfach Fellows'' as well as
from NSF grant DMS--1601024 and EPSCoR grant OIA--1557417.
Szpond was partially supported by the National Science Centre, Poland, grant
2014/15/B/ST1/02197.
Van Tuyl acknowledges the
financial support provided by NSERC RGPIN-2019-05412.

%%%%%%%%%%%%%%%%%%%%%%%%%%%%%%%%%%%%%%%%%%%%%%%%%%%%%%%%%%%%%%%%%%%%%

\section{Background results}
\label{sec:res}

The purpose of this section is two-fold.  We first weave together
the various strands in the literature to present a 
proof of Theorem \ref{thm:all powers equal}, which is a graded version
of known results.  We then give a description of the graded minimal free
resolution of $I^m$ under suitable hypotheses on $I$, again using
known results in the literature.

For the convenience of the reader, we first recall the relevant
definitions.  Let $R = k[x_0,\ldots,x_n]$  and denote by $\fm$ the
homogeneous  maximal
ideal $\fm=(x_0,\ldots,x_n)$. For any homogeneous ideal $I \subseteq R$,
the {\it saturation} of $I$ is the
ideal defined by
$I^\sat=\bigcup_{k=1}^\infty I:\fm^k.$ We say that an ideal $I$
is saturated if $I=I^\sat$.
A homogeneous ideal $I \subseteq R$ is {\it Cohen-Macaulay} (or {\it
perfect}) if ${\rm depth}(R/I) = \dim (R/I)$.

In the following, we denote by $I_X$ the saturated homogeneous ideal
defining a projective scheme $X$.
A subscheme $X$ of $\mathbb P^n$ is {\em arithmetically Cohen-Macaulay (ACM)}  or {\em arithmetically Gorenstein}
if $R/I_X$ is a Cohen-Macaulay ring or a Gorenstein ring, respectively.  The {\it codimension} of $X$ is
${\rm codim(X)} = n - \dim(X)=\h(I_X)$.
The subscheme $X$ is an {\em almost complete intersection}
if the number of minimal generators is one more than
the codimension.
A subscheme $X$  is {\em locally a complete intersection} if the
localization of $I_X$ at any
prime ideal $\fp$ such that $\mathfrak p\neq \fm $ and $I_X\subseteq \fp$
is a complete intersection of codimension equal to the codimension of $X$.
A subscheme $X$  is a {\em generic complete intersection} if the
localization of $I_X$ at any
minimal associated prime ideal of $I_X$
is a complete intersection of codimension equal to the codimension of $X$.
Finally, a scheme $X$ is {\it equidimensional} if $I_X$ is an unmixed ideal,
that is, all of the associated primes of $I_X$ have the same height.

The next result is part of the folklore; we include a short
proof for completeness.

\begin{lemma} \label{LCI}
Let $X \subset \mathbb P^n$ be locally a complete
intersection scheme of any codimension. Then $I_X^{(m)}$ is equal
to the saturation of $I_X^m$. \end{lemma}

\begin{proof}
We denote for brevity $I=I_X$. Let $m$ be a positive integer and 
$\fp\in \Ass(R/I^m)\setminus\{\fm\}$. Then
$\fp R_\fp\in \Ass(R_\fp/(I^m)_\fp)=\Ass(R_\fp/(I_\fp)^m)=
\Ass(R_\fp/I_\fp)$, where the first equality holds because 
localization commutes with taking powers of ideals and the second holds 
because $I_\fp$ is a complete intersection and powers of complete 
intersections are unmixed (\cite[Appendix 6, Lemma~5]{ZS}). Thus 
$\fp\in \Ass(R/I)$ and the desired conclusion follows because the 
argument above yields that $\Ass(R/I^m)$ is in $\Ass(R/I)\cup\{\fm\}$.
\end{proof}

\begin{remark}
Lemma \ref{LCI} can be rephrased to say that for a locally complete
intersection scheme $X$ with  ideal sheaf $\mathcal I_X$, we have \[
I_X^{(m)} = \bigoplus_{t \geq 0} H^0(\mathcal I_X^m (t)).
\]
Another way to interpret Lemma \ref{LCI} is that the only possible embedded
prime of $I^m$ is the maximal ideal $\mathfrak{m}$.  Hence
$I^m = I^{(m)}$ if and only if ${\rm depth}(R/I^m) > 0$.
\end{remark}

In the following, $\mu(J)$ stands
for the cardinality of a minimal set of generators for an ideal $J$.
 The
 next theorem gives necessary and sufficient conditions for  the
 equality of ordinary and symbolic powers for two families of
ideals in terms of the number of
 generators of an ideal.
 Our statement is more general than the one presented in the introduction
 since we can also say something about codimension three arithmetically Gorenstein
 schemes that are also locally a complete intersection.

\begin{theorem}
\label{thm:gen all powers equal}
Let $I = I_X$ be the saturated homogeneous ideal defining a subscheme $X \subset\P^n$
such that  one of the two sets of assumptions listed below holds:

\begin{minipage}[l]{0.45\textwidth}
\underline{Assumptions I:}
\begin{itemize}[leftmargin=*]
   \item $\codim(X)=2$;
   \item $X$ is arithmetically Cohen-Macaulay;
   \item $X$ is locally a complete intersection
\end{itemize}
\end{minipage}
\quad or \quad
\begin{minipage}[c]{0.45\textwidth}
\underline{Assumptions II:}
\begin{itemize}[leftmargin=*]
   \item $\codim(X)=3$;
   \item $X$ is arithmetically Gorenstein;
   \item $X$ is locally a complete intersection.
\end{itemize}
\end{minipage}
\vspace{.1cm}

\noindent
Then the following conditions are equivalent:
\begin{itemize}
\item[(a)] $I^{(n)} = I^n$; 
   \item[(b)] $I^{(m)}=I^m$ for all $m\geq 1$;
   \item[(c)] $I$ has at most $n$ minimal generators.
\end{itemize}
Furthermore, if $m<n-1$ for Assumptions II and $n$ even, or if $m<n$ in all other cases,
then $I_X^{(m)} = I_X^m$ regardless of the number of generators.
\end{theorem}

\begin{proof}
Recall that for any graded homogeneous ideal $J$
of $R = k[x_0,\ldots,x_n]$, if we localize at the maximal ideal $\mathfrak{m}$,
then $\mu(J) = \mu(J_{\mathfrak{m}})$.  Furthermore,
${\rm depth}(R/J) = {\rm depth}(R_{\mathfrak{m}}/J_{\mathfrak{m}})$ by
\cite[Proposition 1.5.15]{BrunsHerzog}.
Therefore all of our assumptions localize and thus one can reduce to the case where $I$ is an ideal in a regular local ring $R$.

The implication $(a)\Rightarrow (b)$ follows by the proofs of
Morey's results \cite[Theorems 3.2 and 3.3]{Mo} and $(b)\Rightarrow (c)$ follows by \cite[Corollary 3.4]{Mo}.
Next we explain how the last conclusion and the remaining implications follow in a more general context. Note that by \cite[Theorem 2.3]{H0} the classes of ideals considered in this theorem are licci (i.e., in the linkage
class of a complete intersection).

For the remaining implications we reason as follows. Assume that $I$ is a licci ideal  that is locally a complete intersection on the punctured spectrum of $R$ and has $\mu(I)\leq n$.  By \cite[Theorem 1.14]{H0} licci ideals have Cohen-Macaulay Koszul homology for their generating sequences. Then  \cite[Remark 2.10 and Corollary 2.13]{U}     yield
 \[\depth{(R/I^m)}\geq n+1-\mu(I), \ \ \forall m\geq 1.  \]  
 Since $\mu(I)\leq n$, this estimate gives that $\depth{(R/I^m)}\geq 1,\ \forall m\geq 1.$   Therefore the ideal $I^m$ is saturated, and
by Lemma \ref{LCI} we conclude that $I^{(m)}=I^m$.  This gives that $(c)\Rightarrow(b)$. Finally, $(b)\Rightarrow(a)$ is clear.

Working still under the assumption that $I$ is a licci ideal  that is locally a complete intersection on the punctured spectrum of $R$, then \cite[Remark 2.10]{U} gives that $\depth(R/I^j)\geq n+2-\h(I)-j$ for all $j$ such that $1\leq j\leq n+2-\h(I)$.
Assume that $\h(I)=2$. Then for $m < n = n+2-\h(I)$, the depth estimate above gives $\depth(R/I^m)\geq n-m>0$. Again,  the ideal $I^m$ is saturated, and
by Lemma \ref{LCI} we conclude that $I^{(m)}=I^m$.
Now assume that $\h(I)=3$. In this case the depth estimate above yields that $I^{(m)}=I^m$ for $m<n-1$.
When $X$ is an arithmetically Gorenstein variety of codimension 3 (Assumptions II), we may also consider the complex $\mathfrak{D}_{\bullet}$ defined in \cite{KU}. The strand $\mathfrak{D}_m$ of the complex will be a resolution for $I^m$ by \cite[Theorem 6.25]{KU}, which gives ${\rm depth}(R/I^m) =n+1-\min\{\mu(I),1+2\lfloor\frac{m+1}{2}\rfloor\} \geq n-2\lfloor\frac{m+1}{2}\rfloor$. For $m=n-1$ we obtain ${\rm depth}(R/I^{n-1})>0$  if $n$  is odd. This proves the  last statement in our theorem.
 \end{proof}

\begin{remark}
\label{quest:licci}
The two classes of ideals considered in Theorem \ref{thm:gen all powers equal} 
are licci (linked to complete intersections). 
The equivalence of parts $(b)$ and $(c)$ of Theorem \ref{thm:gen all powers equal} remains valid 
in the case of licci ideals $I$. Indeed, the implication $(c)\Rightarrow (b)$ is shown above and $(b)\Rightarrow (c)$ follows from the estimate $\depth(R/I^m) =n+1-\mu(I)$ for $m\gg0$ discussed in the preceding proof.
It would be interesting to 
investigate whether the equivalence between these statements and 
condition (a) remains valid for licci ideals. We pose this problem as  Question \ref{Q1} (iii).
\end{remark}

\begin{remark} \label{PowersEquality}
Cases of Theorem \ref{thm:gen all powers equal} were previously known:

(i) If $X$ is a set of points in $\mathbb P^2$, then $I_X^{(m)}
\neq I_X^m$ for any $m \ge 2$ if and only if $X$ is not a complete
intersection. This follows from \cite[Theorem 2.8]{HU}, which
gives that $R/I_X^m$ is not Cohen-Macaulay for any $m \ge 2$
unless $X$ is a complete intersection.

(ii) If $C \subseteq \mathbb{P}^3$ is a curve that is locally a complete intersection
and is an almost complete intersection, then
Theorem \ref{thm:gen all powers equal}
gives \cite[Corollary 2.7]{P}.

(iii)  A connection between the number of generators 
and the equality of the regular and symbolic powers in a special
case can be found in \cite[Corollary 2.5]{H}.  In particular, if $R$
is a regular local ring with $\dim R = 3$, and if $P$ is
a height two prime ideal with three or more generators, it
is shown that $P^m \neq P^{(m)}$ for all $m \geq 1$.  
See also the discussion of \cite[Remark 1.27]{G}.

(iv)  As we will show in Section 4, any ACM set of points in $\popo$
is also locally a complete intersection.
Theorem \ref{thm:gen all powers equal}
thus implies that if $X$ is  any ACM set of points in $\mathbb
P^1 \times \mathbb P^1$, then we get $I_X^{(2)} = I_X^2$.  This was
first shown in \cite{GHVT}.
See also Corollary \ref{P1xP1 result}.

\end{remark}

\begin{remark}
The {\it big height} of an ideal $I$, denoted ${\rm bigheight}(I)$,
is the maximum among the heights of the minimal primes of $I$.
Huneke asked if $I^{(m)} = I^m$ for all $m \leq  {\rm bigheight}(I)$, then is it true
that $I^{(m)} = I^m$ for all $m \geq 1$?  In \cite{GHVT} it was shown
that the answer to this question is negative by showing that
when $X$ is an ACM set of reduced points in $\popo$, then ${\rm bigheight}(I_X) = 2$, but one needs to check if $I_X^{(3)} = I_X^3$ to guarantee that $I_X^{(m)} = I_X^m$ for all $m \geq 1$.   In fact, Theorem \ref{thm:gen all powers equal} shows that we may have to check powers arbitrarily larger than $ {\rm bigheight}(I)$ to guarantee
that $I^{(m)} = I^m$ for all $m \geq 1$.  
\end{remark}

The remainder of this section is devoted to 
describing the graded minimal free
resolution of $I^m$ when $I$ is a perfect ideal of codimension two
under some additional hypotheses. The length of these complexes determines whether
$I^m=I^{(m)}$ and hence are closely connected to Theorem \ref{thm:gen all powers equal}.

Let us start with some preparation.

\begin{lemma}
\label{lem:finite homology}
Consider a complex of finitely generated $R$-modules
\[
0 \to F_p \stackrel{\partial_p}{\longrightarrow} F_{p-1} \to
\cdots \to F_0 \stackrel{\partial_0}{\longrightarrow} M \to 0,
\]
where the modules $F_0,\ldots,F_p$ are free $R$-modules. If the
complex has homology of finite length, $\partial_0$ is surjective,
and $p \le n$, then the complex is exact.
\end{lemma}

\begin{proof}
Break the complex into short exact sequences and compute
local cohomology, or apply the New Intersection Theorem \cite{PS}.
\end{proof}

Recall that a homomorphism $\ffi: F \to G$ of free $R$-modules is
called \emph{minimal} if its image is contained in $\fm G$.
 This means that any coordinate matrix of $\ffi$ has no unit entries.

\begin{lemma}
\label{lem:EN-complex}
Let $I \subset R$ be a homogeneous ideal  admitting a free graded presentation
\[
 F \stackrel{\ffi}{\longrightarrow} G \to I \to 0.
\]
Then, for every integer $m > 0$,  there is a complex of graded
$R$-modules
\[
\begin{array}{l}
0 \rightarrow \bigwedge^{m}  F \to \bigwedge^{m-1} F \otimes
\Sym^{1}  G \rightarrow \bigwedge^{m-2}  F \otimes {\Sym}^{2}
G \rightarrow \cdots \hspace{2in}   \\ \\

\hfill \rightarrow \bigwedge^2  F \otimes {\Sym}^{m-2}  G
\rightarrow  F \otimes {\Sym}^{m-1}  G \rightarrow {\Sym}^m  G
\rightarrow {\Sym}^m I \rightarrow 0
\end{array}
\]
whose right-most map is surjective. Moreover, if $\ffi$ is a minimal map, then all the
maps in the complex above are minimal.
\end{lemma}

\begin{proof}
Let $\Sym I$  denote the symmetric algebra of $I$. A presentation for $\Sym I$  can be
obtained from the given presentation of $I$  by applying the symmetric algebra functor,
which yields the exact sequence $$
F \otimes \Sym G  {\longrightarrow} \Sym G \to \Sym I \to 0.$$
Concretely, if $\rk G=r$ and $t_1,\ldots, t_r$ are new indeterminates, then
the surjection $\Sym G\cong R[t_1,\ldots, t_r] \to \Sym I$ is obtained by mapping each
of the variables $t_i$ to a generator $f_i$ of $I$. We view this as a bigraded map,
by assigning $\deg(t_i)=(1,\deg(f_i))$ and declaring that each element  $g\in R$ has
bidegree $(0,\deg(g))$ in $\Sym  G$. We shall refer to the first component of this
bigrading as the $t$-degree.
If $\rk F=s$, then the kernel of the map ${\Sym}\ G \to {\Sym}\ I$ is the ideal
$C=\left(\sum_{i=1}^r \ffi_{ij}t_i ~|~ 1\leq j\leq s\right)$, where $ \ffi_{ij}$ are the
entries of a coordinate  matrix representing $\ffi$. In this notation, the short exact
sequence above gives $\Sym I \cong (\Sym  G)/C$.

The Koszul complex $\bf{K}_\bullet$ on the generators of  the ideal $C$ of $\Sym G$
takes the following form:
$$
0 \rightarrow \bigwedge^{s}  F(-s) \otimes \Sym G  \rightarrow \bigwedge^{s-1}
F(-s+1) \otimes \Sym G\rightarrow \cdots \rightarrow\bigwedge^{1}  F(-1) \otimes
\Sym G  \rightarrow \Sym \ G \rightarrow 0.
$$
It is a  complex of free $\Sym G$-modules and the graded twists refer to the
$t$-grading.
The linear strand in $t$-degree $m$ of the Koszul complex is the following complex of
free $R$-modules, which also appears in  \cite[p.\ 597]{Ei}:
\[
\begin{array}{l}
0 \rightarrow \bigwedge^{m}  F \to \bigwedge^{m-1} F \otimes_R  \Sym^{1}
 G \rightarrow \bigwedge^{m-2}  F \otimes_R {\Sym}^{2}  G
\rightarrow \cdots \hspace{2in}   \\ \\

\hfill \rightarrow \bigwedge^2  F \otimes_R {\Sym}^{m-2}  G
\rightarrow  F \otimes_R {\Sym}^{m-1}  G \rightarrow {\Sym}^m  G.
\end{array}
\]
Moreover, the given presentation for $\Sym I$ induces the exact sequence
\[
F \otimes  {\Sym}^{m-1}  G \rightarrow {\Sym}^m  G \rightarrow {\Sym}^m I
\rightarrow 0.
\]
Combining the two complexes gives the desired conclusion. The differentials
in the family of complexes described above involve only the elements $\ffi_{ij}$,
thus minimality for any of these complexes is equivalent to the minimality of $\ffi$.
\end{proof}

The above family of complexes can be used to extract information on the minimal free
resolution for the powers of $I$ in several cases. We introduce the
notation $V(J)=\{\fp \in {\rm Proj}(R)\; \vert \; J\subseteq \fp, \fp\neq \fm\}$  for the elements of the
 punctured spectrum of $R$ containing $J$.

\begin{theorem}
   \label{thm:EN exact}
Consider a graded minimal free resolution of a homogeneous perfect
ideal $I \subset R = k[x_0,\ldots,x_n]$ of codimension two
\[
0 \to F \stackrel{\ffi}{\longrightarrow} G \to I \to 0.
\]
Let $m$ be a positive integer and assume further that $I$ is locally a complete intersection and
$\min \{ \rk G - 1, \ m\} \le n$.
Then $\Sym^m I \cong I^m$ and  the complex in Lemma \ref{lem:EN-complex} is a graded minimal
free resolution of $I^m$.
\end{theorem}

\begin{proof}
In view of our Lemma \ref{lem:EN-complex}, it is sufficient to verify that the complex
therein is acyclic and that the canonical surjection $\Sym^m I\to I^m$ is in fact an
isomorphism  $\Sym^m I\cong I^m$.  We continue with the notation introduced in the
proof of Lemma \ref{lem:EN-complex}. 

Localizing the short exact sequence $0\to F\to G \to I\to 0$ at $\fp\in V(I)$ yields
the direct sum of a minimal free resolution for the height two complete intersection
 $I_\fp=(f,g)$ and an isomorphism $0\to R_\fp^{s-1}\to R_\fp^{s-1}\to 0$. This gives
$C_{\fp R[t_1,\ldots, t_s]}=(ft_1-gt_2, t_3,\ldots,t_s)$, thus $C$ is a complete
intersection in $R[t_1,\ldots t_s]$. Similarly, when localizing at  primes $\fp$
not containing $I$,  $\mathbf{K}_\bullet$ becomes the Koszul complex on the
variables $t_1,\ldots, t_s$ and thus $\mathbf{K}_\bullet$ and all of its graded strands
are exact complexes when localized at $\fp\neq \fm$. Therefore the homology of
the graded strands of $\mathbf{K}_\bullet$ has finite length. Applying Lemma
\ref{lem:finite homology}, we  conclude that the complex in Lemma \ref{lem:EN-complex}
is exact when $m\le n$.

Furthermore, by \cite[Theorem 5.1]{Tc}, the isomorphism $\Sym^m I\cong I^m$ holds
true if $\mu(I_\fp)\leq \depth R_\fp$  for all prime ideals $\fp$ containing $I$
such that  $\depth R_\fp \leq \min \{ \mu(I), \ m\}$.   For  $\fp\in V(I)$, using
the fact that $I$ is locally a complete intersection, we have
$\mu(I_\fp)=\h(I_\fp)=\depth I_\fp\leq \depth R_\fp$. Next we analyze the possibility
that $\fp=\fm$ is among the primes that satisfy $\depth R_\fp \leq \min \{ \mu(I), \ m\}$.
This occurs when $n+1=\depth R_\fm \leq \min \{ \mu(I), \ m\}=\min \{ \rk G, \ m\}$.
Since by hypothesis we have $\min \{ \rk G - 1, \ m\}\leq n$,  it must be the case
that $\mu(I)=\rk G=n+1\leq m$. But in this case, $\mu(I_\fm)=\rk G=n+1=\depth R_\fm$,
thus the desired conclusion that  $\Sym^m I\cong I^m$ follows.
 \end{proof}

 \begin{remark}\label{altapproach}
 The proof of \cite[Theorem 5.1]{Tc} shows that the locally complete intersection
hypothesis in Theorem \ref{thm:EN exact} can be weakened to  $\mu(I_\fp)\leq \depth R_\fp$
for all primes containing $I$ and such that $\depth R_\fp \leq \min \{ \mu(I), \ m\}$.
\end{remark}

\begin{remark}\label{HSVremark}
Previously known approaches to obtaining Theorem  \ref{thm:EN exact}:

(i) Lemma \ref{lem:EN-complex} and 
Theorem \ref{thm:EN exact} could also be obtained by applying the
results of \cite{HSV}.  In particular, the complex of Lemma \ref{lem:EN-complex} is the
{\em approximation complex} of \cite{HSV} constructed from the minimal free resolution
of a perfect ideal of codimension two. We prefer to give a more explicit description of this complex for future use in section \ref{sect:Romer}.

(ii) In \cite[Theorem 5.4]{ABW}, the resolutions of powers of an ideal that
satisfies $\h I_j (\ffi) \ge \mu(I) + 1 - j$ are given. This provides an alternate
proof of our Theorem \ref{thm:EN exact}.
\end{remark}

It is of interest to record here
the minimal free resolution that we obtain in the special cases of
$m=2$ and $m=3$ for ACM curves in $\mathbb P^3$. This will be useful in section \ref{sec:applications}.

\begin{corollary}
\label{cor:res powers}
Let $X$ be an arithmetically Cohen-Macaulay curve in $\mathbb P^3$
that is locally a complete intersection, and let
$
0 \to F \stackrel{\ffi}{\longrightarrow} G \to I_X \to 0
$
be a graded minimal free resolution.
Then one has:
\begin{itemize}
\item[(a)] The graded minimal free resolution of $I_X^2$ is
\[
0 \to \bigwedge^2  F    \rightarrow  F \otimes  G \rightarrow {\Sym}^2  G
\rightarrow I_X^2 \rightarrow 0.
\]

\item[(b)]The graded minimal free resolution of $I_X^3$ is
\[
0 \to \bigwedge^3  F \to \bigwedge^2  F \otimes G \rightarrow
 F \otimes {\Sym}^{2}  G \rightarrow {\Sym}^3  G \rightarrow I_X^3 \rightarrow 0.
\]
\end{itemize}
\end{corollary}

\begin{remark} \label{connection}
(i)
The above result applies in particular to any smooth space curve that is
arithmetically Cohen-Macaulay.

(ii) Note that any arithmetically Cohen-Macaulay union of lines in
$\mathbb{P}^3$
such that no three lines meet in a point is locally a complete intersection.
Again, the above result applies to configurations of this type.
 \end{remark}

%%%%%%%%%%%%%%%%%%%%%%%%%%%%%%%%%%%%%%%%%%%%%%%%%%%%%%%%%%%%%%%%%%%%%

\section{Remarks on the hypotheses}\label{sec:remarks}

In this section we comment on the importance of the various hypotheses in our
results in Section \ref{sec:res},  especially Theorem \ref{thm:gen all powers equal},
and we give examples to indicate various ways that they might be weakened.

\subsection{The ACM hypothesis}
The assumption that $X$ is arithmetically Cohen-Macaulay is essential in
Theorem \ref{thm:gen all powers equal}. If it is dropped, then it might happen that
all symbolic and ordinary powers
of $I$ are equal but $I$ has more than $n$ generators.
This is illustrated by Example \ref{ex:5 lines}.

\begin{example}\label{ex:5 lines}
Let $X$ be the union of $5$ lines in $\P^3$ defined by 5 general points in $\P^1\times\P^1$
(see Section 4 for more on points in $\P^1 \times \P^1$).  This configuration has
codimension two and is locally a complete intersection, but is not ACM.
However, from \cite[Theorem 3.1.4]{GHVT2}, $I=I_X$ has $6$ minimal generators and $I^{(m)}=I^m$ for all $m\geq 1$.
A similar phenomenon  holds for $s= 2$, or $3$ general points in
$\mathbb P^1 \times \mathbb P^1.$
The case of $s=2$  general points (i.e., two skew lines in $\mathbb P^3$ or even in
$\mathbb P^n$) is also discussed in \cite[Remark~4.2]{GHVT3}. Note that
all the cases of this example are non-ACM cases.
\end{example}

It is interesting to notice that for $5$ general lines in $\P^3$ the picture
is quite different, but still shows that the ACM hypothesis is necessary
for the final statement of Theorem \ref{thm:gen all powers equal}.

\begin{example}\label{ex:5 general lines}
Let $X$ be the union of $5$ general lines $L_1,\ldots,L_5$ in $\P^3$.   Again,
$X$ is not ACM, but ${\rm codim}(X) = 2$ and $X$ is locally a complete intersection.
Then for $I=I_X$ we have
$$I^{(2)}\neq I^2.$$
Indeed, for any three mutually distinct indices $\left\{i,j,k\right\}\subset\left\{1,2,3,4,5\right\}$
with $i<j<k$ let $Q_{ijk}$ be the quadric containing the lines $L_i$, $L_j$ and $L_k$.
Altogether there are $10$ such quadrics. The ideal $I$ is then generated by the following
products of the quadrics
$$F_1=Q_{123}Q_{145},\; F_2=Q_{145}Q_{235},\; F_3=Q_{235}Q_{124},\; F_4=Q_{124}Q_{345},\; F_5=Q_{345}Q_{123},$$
$$F_6=Q_{125}Q_{234},\; F_7=Q_{234}Q_{135},\; F_8=Q_{135}Q_{124},\; F_9=Q_{134}Q_{245},\; F_{10}=Q_{345}Q_{125}.$$
These generators not only vanish along all lines but they vanish along one of the lines
to order $2$. Since there is a cubic vanishing along $4$ general lines in $\P^3$, this
gives rise to $10$ septics vanishing to order $2$ along all five lines. Hence there
are elements of degree $7$ in $I^{(2)}$ but the initial degree of $I^2$ is $8$.
This example shows that the final statement of Theorem\ref{thm:gen all powers equal} also
requires the ACM hypothesis.
\end{example}

Nevertheless, the  condition that $X$ be ACM  in Theorem ~\ref{thm:gen all powers equal}
seems not to be essential if we allow the codimension to go up. Indeed, it is  of interest
to seek results about symbolic powers for non-ACM subschemes, possibly of higher
codimension. We first give a simple such result.

\begin{theorem}\label{thm: complete intersections codim4}
Let $C = C_1 \cup C_2$ be a disjoint union of two complete
intersections of dimension $r$ in $\mathbb P^{2r+1}$. Then $I_C^m = I_C^{(m)}$ for all positive integers $m$.
\end{theorem}

\begin{proof}
For $i=1,2$, we have that $I_{C_i}^m = I_{C_i}^{(m)}$ for all $m$ since $I_{C_i}$ is generated by
a regular sequence (see \cite[Lemma 5, Appendix 6]{ZS}).  Thus
\[
I_C^m = (I_{C_1}I_{C_2})^m = I_{C_1}^m \cdot I_{C_2}^m = I_{C_1}^m
\cap I_{C_2}^m =  I_{C_1}^{(m)} \cap I_{C_2}^{(m)} = I_C^{(m)}.
\]
The first and third equalities are true because for disjoint ACM subschemes
of dimension $r$ in $\mathbb P^{2r+1}$, the intersection of the
ideals is equal to the product thanks to  a special case of Th\'eor\`eme 4 and
 the subsequent Corollaire in \cite[pp. 142--143]{serre}. \end{proof}

 The following example shows that even the assumption that $C$ be a disjoint union, in Theorem~\ref{thm: complete intersections codim4},  is not always
 needed.

\begin{example}\label{remark:two planes in P^4}
Consider the union, $C$, of two planes in
$\mathbb P^4$ meeting at a point $P$. At $P$, $C$~not only fails
to be locally a complete intersection, but in fact it fails to be
locally Cohen-Macaulay, since $C$ is a cone over two skew lines in
$\mathbb P^3$ with vertex at $P$. Yet a similar argument as given in the proof of Theorem \ref{thm: complete intersections codim4} shows that $I_C^m =I_C^{(m)}$ for all $m \geq 1$.
In particular, it is worth noting that  the powers of $I_C$ do not pick up an embedded point at~$P$.
\end{example}

\subsection{The hypothesis on low codimension}
As noted in Remark \ref{quest:licci} It is conceivable that the equivalence of the three statements (a), (b) and (c)
in  Theorem ~\ref{thm:gen all powers equal}  may hold  in higher codimension as long as the ideal is locally a complete intersection and linked to a complete intersection. We offer some evidence towards this conjecture below.
In the following example we give ``essentially" the same licci ideal,
viewed in $\mathbb P^3$ and in $\mathbb P^4$. In both cases the ideal
is locally a complete intersection (in fact smooth) and is ACM.
It has four minimal generators,  so the condition that $I$ has at most $n$ minimal generators
fails in the case of $\mathbb P^3$ but is satisfied in the case of $\mathbb P^4$.
This example was obtained using CoCoA~\cite{cocoa}.

\begin{example} \label{ex:P3P4}
Consider a sufficiently general complete intersection of type (1,1,2) in $\mathbb P^3$
and (separately) also in $\mathbb P^4$. In each case, link this ideal using a sufficiently
general complete intersection of type $(2,2,2)$. Both in $\mathbb P^3$ and in $\mathbb P^4$,
the residual will be ACM of codimension 3, and an easy mapping cone argument shows that
the residual has 4 minimal generators and Hilbert function $(1,4,6,6,\dots)$.  In $\mathbb P^3$ the
residual is a set of 6 distinct points, and in $\mathbb P^4$ the residual is a smooth
curve of degree 6 and genus 2, so both are locally complete intersections. Indexing
by the ambient space, we denote these ideals by $I_3$ and $I_4$ respectively. One can
check by hand or with a computer algebra program that the Betti diagram of both $I_3$ and $I_4$ is
\[
\begin{array}{c|cccccccc}
& 0  &  1 &   2 &   3 \\ \hline
 0: &    1 &   - &   -  &   - \\
 1:  &   - &   4  &  2 &   - \\
 2:  &   -  &  -  &  3 &   2 \\ \hline
\hbox{Tot:}  &  1  &  4 &   5 &   2
\end{array}
\]
Then one can check with a computer algebra program that the
Betti diagram of both $I_3^2$ and $I_4^2$ is 
\[
\begin{array}{c|cccccccccc}
    &    0 &   1 &   2  &  3  &  4 \\ \hline
 0: &    1 &   - &   - &   - &   - \\
 1: &    - &   -  &  -  &  - &   - \\
 2:  &   -  &  -  &  -  &  -  &  - \\
 3: &    - &  10  &  8  &  1 &    - \\
 4: &    - &   -  &  9 &   8 &   1 \\ \hline
\hbox{Tot:} &   1 &  10 &  17  &  9 &   1
\end{array}
\]
By the Auslander-Buchsbaum formula, this means that $I_3^2$ is not saturated,
hence $I_3^2 \neq I_3^{(2)}$. On the other hand, it means that $I_4^2$ is saturated;
then since $I_4$ is locally a complete intersection, we get $I_4^2 = I_4^{(2)}$
(although $I_4^2$ is not ACM). We have verified on CoCoA that in
fact $I_4^k = I_4^{(k)}$ for $1 \leq k \leq 7$.
\end{example}

An additional piece of evidence that the hypothesis on the codimension might be weakened
is provided by ideals of scrolls $\P^1\times\P^n\hookrightarrow\P^{2n+1}$. Note that the ideals defining these scrolls are not even licci, as can be shown by applying the criterion in \cite[Theorem 2.8]{U2}.

\begin{example}[Scrolls]\label{ex:scrolls}
Let $X$ be the scroll $\P^1\times\P^n\hookrightarrow\P^{2n+1}$ embedded by the Segre embedding.
The ideal $J$ of $X$ is determined by the $2\times 2$ minors of a generic $2\times(n+1)$ matrix.
By \cite[Corollary 7.3]{DCEP} we have $J^{(m)}=J^m$ for all $m\geq 1$. Indeed, in the
notation of the just cited corollary, we have $J=I_2$, $n=2$ and $\ell=2$.

Notice that the ideal $J$ has $\binom{n+1}{2}$ minimal generators and the
corresponding scroll lies in $\mathbb P^{2n+1}$. The variety we get has codimension $n$.
When $n = 2$ or $n = 3$, it is true that the number of minimal generators is smaller
than the dimension of the projective space. The latter, in particular, gives further
evidence for Question \ref{Q1} below. But as soon as $n \geq 4$, we get an example
where the number of minimal generators is larger than the dimension of the projective
space (violating part (c) of Theorem ~\ref{thm:gen all powers equal}) but
nevertheless the statements of (a) and (b) are true.
\end{example}

With the above examples and comments in mind, we pose the following questions:

\begin{question} \label{Q1}
Let $X \subseteq \P^n$ be an ACM subscheme that is locally a complete intersection.
\begin{enumerate}
\item[$(i)$]  If ${\rm codim}(X) = 3$, are conditions (a), (b) and (c) of
Theorem \ref{thm:gen all powers equal} still equivalent?
\item[$(ii)$] If ${\rm codim}(X) > 3$, are conditions (a) and  (b) of
Theorem \ref{thm:gen all powers equal} still equivalent?
\item[$(iii)$] If $I_X$ is linked to a complete intersection, are conditions (a),(b) and (c) of
Theorem \ref{thm:gen all powers equal} still equivalent?
\end{enumerate}
\end{question}

\subsection{The hypothesis of being locally a complete intersection}
We now address the assumption that $X$ is locally a complete intersection.
For many of the results in this paper, Lemma \ref{LCI} has been
important. It says that the property that a reduced subscheme $X
\subset \mathbb P^n$ is locally a complete intersection implies that
its powers do not pick up non-irrelevant embedded components, so that the failure
of a power to be a symbolic power comes only from the failure to
be saturated. It is natural to ask if the converse holds; that is,
one might ask if the following statement is true: If $X$ is not 
locally a complete intersection, then $I_X^m$ defines a scheme with
embedded components.

We have seen in Example \ref{remark:two planes in P^4} that this is not the case.
On the other hand, it is not hard to verify that if $X$ consists of three lines in
$\mathbb P^3$ meeting at a point,
then $X$ is ACM but $I_X^2$ {\em does} have an embedded point;
so even though $I_X^2$ is saturated, it is not equal to $I_X^{(2)}$.
We make the following conjecture, which is also based on other computer experiments.
Example \ref{remark:two planes in P^4} shows that it is not true without the assumption
that $X \subset \mathbb P^3$.

\begin{conjecture}\label{conj:not LCI}
Let $X\subset\P^3$ be a subvariety (reduced and unmixed) of codimension~2. Assume that there is a point $P \in X$ such that the localization of $I_X$ at $P$ is not a complete intersection. Then for any  integer $m \geq 2$, the saturation of $I_X^m$ has an embedded component at $P$.  In particular,
\[
I_X^{(m)} \neq I_X^m~~\mbox{for all $m\geq 2$.}
\]
\end{conjecture}

The above conjecture has been proved in \cite[Theorem 4.4]{As} for irreducible varieties $X$.
We now give some examples related to this conjecture.

\begin{example}\label{ex:lci 1}
Recall that a fat points scheme in $\mathbb P^n$ is defined by an ideal of the form $J = I(P_1)^{m_1} \cap \cdots \cap I(P_s)^{m_s}$, where $\{P_1, \ldots, P_s\} \subseteq \mathbb P^n$ is a finite set of distinct points and $m_1, \ldots, m_s$ are non-negative integers. We denote the fat points scheme by $m_1P_1 + \cdots + m_sP_s$.  
 By Remark \ref{PowersEquality}(i), if $J$ is the ideal of a reduced set of points
in $\P^2$ (so $J$ is radical), then $J^{(m)}=J^m$ for all $m\geq 1$ (if and)
only if $J$ is a complete intersection.  However, the situation is more subtle
for fat points. There are examples of nonreduced ideals $J = I(P_1)^{m_1} \cap \cdots \cap I(P_s)^{m_s}$ of fat points all of whose powers are symbolic.  Since a nonradical fat points ideal is never locally a complete intersection, this shows that Conjecture \ref{conj:not LCI} is false without the assumption that $X$ is a variety,
where variety here means any reduced subscheme of $\P^n$ (not necessarily irreducible).

It is an interesting but open problem to classify those ideals $I$ of fat points in
$\P^2$ whose powers are all symbolic.
Of course, if $J=I^{(t)}$ where $I$ is a radical complete
intersection ideal of points in $\P^n$, then $J^{(m)}=I^{(tm)}=I^{tm}=(I^t)^m=J^m$,
so $J$ is a fat points ideal and all powers of $J$ are symbolic.
A sufficient condition for $J^{(m)}=J^m$ for all $m\geq 1$ to hold
for an ideal $J$ of fat points in $\P^2$ is given in \cite[Proposition 3.5]{HaHu13}:
if $J$ is of the form $J=I^{(t)}$ where $I$ is a radical ideal of $n$ points in $\P^2$, and if
$\alpha(J)\beta(J)-t^2n=0$, where $\alpha(J)$ is the least degree of a non-zero
form in $J$ and $\beta(J)$ is the least degree $d$ such that the base locus of
the linear system $J_d$ is 0-dimensional, then all powers of $J$ are symbolic.
So for example, if $I$ is either the ideal of five general points in $\P^2$ or
the ideal of a star configuration in $\P^2$, then $J^{(m)}=J^m$ holds for all
$m\geq 1$ for $J=I^{(2)}$ (see \cite[Corollary 3.9 and Lemma 3.11]{HaHu13}; see also \cite{GHM}).
In neither case is $I$ a complete intersection; in particular $I^2\neq I^{(2)}$ for
these examples.
\end{example}

\begin{example}\label{ex:lci 2}
We now recall examples of radical ideals $I(m)$ of  points in $\P^2$ given in \cite{NS}.
The results of \cite{NS} show that the powers of $I(m)^{(m)}$ are all symbolic
but $I(m)^{(m)}$ is not a power of a complete intersection and (as long as $m>3$)
the criterion given in \cite[Proposition 3.5]{HaHu13} does not apply.
(For example, for $m=4$, we have $\alpha(I(4))=16$ and $\beta(I(4))=20$, but $m^2n=16(19)$.)
Fix some integer $m \ge 3$ and consider the ideal
\[
I(m)=(x(y^m-z^m),y(z^m-x^m),z(x^m-y^m))\subset R= k[x,y,z].
\]
This ideal is the homogeneous ideal of  a set of $m^2 + 3$ points, which has been dubbed a
Fermat point configuration. By Theorem \ref{thm:gen all powers equal} (c),
we know $I(m)^{(2)} \neq I(m)^2$.  However,
for each
positive integer $j$ one has by \cite[Proposition 4.1]{NS}
\[
I(m)^{(m j)} = (I(m)^{(m)})^j.
\]
\end{example}

\begin{example}\label{ex:lci 3}
The examples given above of ideals $J = I(P_1)^{m_1} \cap \cdots \cap I(P_s)^{m_s}$ of fat points whose powers are all symbolic
all have the property that $\operatorname{gcd}(m_1,\ldots,m_s)\neq 1$, but this is not essential.
Examples with $\operatorname{gcd}(m_1,\ldots,m_s)=1$ are given in \cite[Example 5.1]{BH}.
In particular, let $J$ be the ideal of $Z = (d-1)P_1 + P_2 + \cdots + P_{2d}$
where $P_1, \ldots, P_{2d}$ are general points in
$\P^2$ and $d>2$. It is shown in \cite{BH} that $J^m = J^{(m)}$ and noted there that there is no
$m \geq 1$ such that $J^{(m)}$ is a power of any ideal which is prime, radical, or a complete intersection.
As noted in \cite{BH}, further examples can be obtained from this using the action of the Cremona group.
These examples, Example \ref{ex:lci 1} above, and the case $m=3$ of Example \ref{ex:lci 2} have the property that $\alpha(J)\beta(J)-\sum_im_i^2=0$, where $\alpha(J)$ and $\beta(J)$ are as defined in Example \ref{ex:lci 1}.  Thus, perhaps \cite[Proposition 3.5]{HaHu13} can be generalized
to ideals of fat points rather than just for certain powers of certain radical ideals.  However
Example \ref{ex:lci 2} shows this would still not cover all cases of ideals of fat points whose powers are all symbolic.
\end{example}

%%%%%%%%%%%%%%%%%%%%%%%%%%%%%%%%%%%%%%%%%%%%%%%%%%%%%%%%%%%%%%%%%%%

\section{Application 1:
Points in $\mathbb P^1 \times \mathbb P^1$}\label{sec:applications}

In this section we apply the main results of Section 2, namely
Theorem \ref{thm:gen all powers equal}, to ACM sets of points
in $\popo$.  In particular, we show how our new results
give new short proofs to results in \cite{GHVT,GVT2}.

We begin with a quick review
of the relevant definitions and notation.   For a more thorough
introduction to this topic, see \cite{GVBook}.
The polynomial ring $k[x_0,x_1,x_2,x_3]$ with the bigrading
given by $\deg x_0 = \deg x_1 = (1,0)$ and $\deg x_2 = \deg x_3 =
(0,1)$ is the coordinate ring of $\P^1 \times \P^1$.
A point $P = [a_0 : a_1] \times [b_0 : b_1]$
in $\mathbb{P}^1 \times \mathbb{P}^1$ has a bihomogeneous
ideal $I_P = (a_1 x_0 - a_0 x_1, b_1 x_2 - b_0 x_3)$.   A set of
points $X = \{P_1,\ldots,P_s\} \subseteq \P^1 \times \P^1$
is associated to the
bihomogeneous ideal $I_X = \bigcap_{P \in X} I_P$.  If we only
consider the standard grading of this ideal, then $I_X$
defines a union $X$ of lines in $\P^3$.  In order
to apply the results of the previous sections, we first require
the following lemma.

\begin{lemma}\label{rem:points in product as lines}
Let $X \subseteq \P^1 \times \P^1$ be any set of points.  Then
$I_X$ is locally a complete intersection.
\end{lemma}

\begin{proof}
We will consider a point $P \in \mathbb{P}^1 \times \mathbb{P}^1$
as the line in $\mathbb{P}^3$ that is defined by the ideal
$I_P = (a_1 x_0 - a_0 x_1, b_1 x_2 - b_0 x_3)$.
We  now show that the union of lines in $\mathbb P^3$ coming from a union of
points in $\mathbb{P}^1 \times \mathbb{P}^1$ in this way is locally a
complete intersection. There is no problem at a smooth point, so we must
determine how two or more such lines can meet at a single point.
Let $C$ be a union of such lines.

The planes of the form $a_1 x_0 - a_0 x_1$ form the pencil of planes,
$\Lambda_1$, through the line $\lambda_1$ defined by $x_0 = x_1 = 0$.
Similarly, the planes of the form $b_1 x_2 - b_0 x_3$ form the pencil of
planes, $\Lambda_2$, through the line $\lambda_2$ defined by
$x_2 = x_3 = 0$. Notice that $\lambda_1$ and $\lambda_2$ are disjoint.
A point $P$ in $\mathbb P^1 \times \mathbb P^1$, then, corresponds to the
line of intersection of a plane in $\Lambda_1$ and a plane in $\Lambda_2$.
Given any point $A$ in $\mathbb P^3$ not on either $\lambda_1$ or $\lambda_2$,
there is a unique element of $\Lambda_1$ and a unique element of $\Lambda_2$
passing through $A$. Hence two lines of $C$ cannot meet at a point not on one
of the two lines, $\lambda_1$ or $\lambda_2$. Now assume that
$A \in \lambda_1$ (the case $A \in \lambda_2$ is identical). In order for
two or more lines of $C$ to meet at $A$, it must be that  the plane
$H_A \in \Lambda_2$ containing $A$ is fixed, while the plane in
$\Lambda_1$ is not. Hence lines meeting at $A$ all lie on $H_A$, and so are
coplanar. But any plane curve is a complete intersection, so the same holds
for any localization. Therefore, any set of points in
$\mathbb{P}^1 \times \mathbb{P}^1$ defines a union of lines in
$\mathbb P^3$ that is locally a complete intersection.
\end{proof}

First we give a short proof of  \cite[Theorem 1.1]{GHVT}.

\begin{theorem} \label{GHVT thm}
Let $X \subseteq \mathbb{P}^1 \times \mathbb{P}^1$ be an ACM set
of points.  Then $I_X^m = I_X^{(m)}$ for all $m \geq 1$ if and
only if $I_X^3 = I_X^{(3)}$.
\end{theorem}

\begin{proof}
We will view $X$ as a union of lines in $\mathbb P^3$. By Lemma
\ref{rem:points in product as lines}, $X$ is locally a complete
intersection.   Now apply Theorem \ref{thm:gen all powers equal}.
\end{proof}

In \cite{GHVT}, it was also asked what sets of points $X \subseteq
\mathbb{P}^1 \times \mathbb{P}^1$ satisfy $I_X^3 =
I_X^{(3)}$ and, in particular, if there is a geometric classification of such points.
The authors proposed such a classification.  We require
the following notation. Let $\pi_1:\P^1 \times \P^1
\rightarrow \P^1$ denote the natural projection onto the first factor.
%\[P = A \times B \mapsto A.\]
If $X \subseteq \P^1 \times \P^1$ is a finite set of reduced points,
$\pi_1(X) = \{A_1,\ldots,A_h\}$ is the set of distinct first
coordinates that appear in $X$.  For $i=1,\ldots,h$,
set $\alpha_i = |X \cap \pi_1^{-1}(A_i)|$, i.e., the number
of points in $X$ whose first coordinate is $A_i$.  After
relabeling the $A_i$'s so that $\alpha_i \geq \alpha_{i+1}$ for
$i = 1,\ldots,h-1$, we set $\alpha_X = (\alpha_1,\ldots,\alpha_h)$.

\begin{remark}
One of the themes of the monograph \cite{GVBook} is to demonstrate that
when $X$ is an ACM set of points in $\P^1 \times \P^1$,
 many of the homological invariants of $I_X$, e.g.,
bigraded Betti numbers, Hilbert function, can be computed
directly from the tuple $\alpha_X$.    As shown in the next corollary,
$\alpha_X$ can also be used to determine when $I_X^{(m)} = I_X^m$.
\end{remark}

We now prove \cite[Conjecture 4.1]{GHVT}.  In fact,
we give a slightly stronger version.

\begin{corollary} \label{P1xP1 result}
Let $X \subseteq \mathbb{P}^1 \times \mathbb{P}^1$ be any ACM
set of points.  Then
\medskip
\begin{itemize}
\item[(a)] $I_X^2 = I_X^{(2)}$.
\medskip
\item[(b)] The following are equivalent:
\medskip
\begin{itemize}
\item[(i)] $I_X^2$ defines an ACM scheme;
\item[(ii)] $I_X^3 = I_X^{(3)}$ is the saturated ideal of an ACM scheme;
\item[(iii)] $X$ is a complete intersection;
\item[(iv)] $\alpha_X = (a,a,\ldots,a)$ for some integer $a \geq 1$.
\end{itemize}
\medskip
\item[(c)] The following are equivalent:
\medskip
\begin{itemize}
\item[(i)] $I_X^3 = I_X^{(3)}$ is the saturated ideal of a non-ACM scheme;
\item[(ii)] $I_X$ is an almost complete intersection;
\item[(iii)] $\alpha_X = (a,\ldots,a,b,\ldots,b)$ for integers $a > b \geq 1$.
\end{itemize}
\end{itemize}
\end{corollary}

\begin{proof}
Since $X$ (viewed as a union of lines in $\mathbb P^3$) is locally a
complete intersection, for any $m$ the condition that $I_X^m =
I_X^{(m)}$ is equivalent to the condition that $I_X^m$ is
saturated., Equivalently, $\operatorname{proj-dim}(I_X^m)\leq 2$. In the former case
the scheme it defines is ACM; in the latter case it is not. Part
(a) was first proved in \cite[Theorem 2.6]{GHVT}, but it also
follows from Corollary \ref{cor:res powers} (a). From Corollary
\ref{cor:res powers} (a) we see that the scheme defined by
$I_X^2$ is ACM if and only if $\bigwedge^2 F = 0$, i.e., ${\rm rank}(F)=1$, meaning that $X$ is a complete intersection.

That $I_X^3$ is the saturated ideal of an ACM scheme
is equivalent by Corollary \ref{cor:res powers} (b) to $\bigwedge^2F=0$, which means again
that $X$ is a complete intersection. The equivalence of (iii) and (iv)
in (b) is \cite[Theorem 5.13]{GVBook}.
This proves (b).

To prove (c) we need the resolution to be one step longer, which is
equivalent to $\bigwedge^2 F \neq 0$ and $\bigwedge^3 F = 0$,
i.e., $F$ has rank 2, meaning that $I_X$ has three minimal generators.
By \cite[Corollary 5.6]{GVBook}, $I_X$ has three minimal generators
if and only if there exist integers $a,b$ with $a> b$ such that
 $\alpha_X = (a,\ldots,a,b,\ldots,b)$.
\end{proof}

Guardo and Van Tuyl give an algorithm
\cite[Algorithm 5.1]{GVT2} to compute
the bigraded Betti numbers for any set of double points in $\P^1 \times
\P^1$ provided that the support is ACM.  Although we will not reproduce
the algorithm here, it was shown that the bigraded Betti numbers of
a set of double points only depend upon the tuple $\alpha_X
= (\alpha_1,\ldots,\alpha_h)$ describing the support $X$.  We
can now deduce this result from our new work.

\begin{corollary}
Let $Z \subseteq \P^1 \times \P^1$ be a set of fat points where
every point has multiplicity two.  If $X$ is the support of $Z$
and if $X$ is ACM, then there exists an algorithm to
compute the bigraded Betti numbers of $I_Z$ using only
$\alpha_X$.
\end{corollary}

\begin{proof}
Because $X$ is ACM, the bigraded minimal free resolution of
$I_X$ can be computed directly from $\alpha_X$ (see
\cite[Theorem 5.3]{GVBook}).  By Corollary \ref{P1xP1 result},
$I_Z = I_X^{(2)} = I_X^2$.  We can then use
Corollary \ref{cor:res powers} to compute the bigraded
resolution of $I_X^2$ using the bigraded resolution of $I_X$.
In particular, the bigraded Betti numbers of $I_Z$ only depend
upon knowing $\alpha_X$.
\end{proof}

Note that one can use Corollary \ref{cor:res powers} to write out
all the bigraded Betti numbers.  Although we will not do this here,
we will show how to compute the bigraded Betti numbers
of triple points whose support is an ACM set of points that is
an almost complete intersection.  Except for
the result in the remark below, we are not aware of any similar results of
this type.

\begin{corollary}
Let $Z \subseteq \P^1 \times \P^1$ be a homogeneous set of triple
points (i.e., where every point has multiplicity three) and let  $X$ denote the support of $Z$.  If $I_X$ is an almost
complete intersection with $\alpha_X =
(\underbrace{a,\ldots,a}_c,\underbrace{b,\ldots,b}_d)$,
then $I_Z$ has a bigraded minimal free resolution of the form
\[0 \rightarrow F_2 \rightarrow F_1 \rightarrow F_0 \rightarrow I_Z \rightarrow 0\]
where
\footnotesize
\begin{eqnarray*}
F_0 &  = & R(-3c-3d,0)\oplus R(-3c-2d,-b) \oplus R(-2c-2d,-a) \oplus R(-3c-d,-2b) \oplus R(-2c-d,-b-a) \oplus \\
&&R(-c-d,-2a)\oplus R(-3c,-3b) \oplus R(-2c,-2b-a) \oplus R(-c,-b-2a) \oplus R(0,-3a)\\
F_1 & = &R(-c,-3a)\oplus R(-2c,-2a-b) \oplus R(-3c,-a-2b) \oplus R(-c-d,-2a-b)   \oplus R(-2c-d,-a-2b) \oplus \\
&&R(-3c-d,-3b) \oplus R(-2c-d,-2a) \oplus R(-3c-d,-a-b)\oplus R(-2c-2d,-a-b) \oplus \\
&& R(-3c-2d,-2b) \oplus R(-3c-2d,-a) \oplus R (-3c-3d,-b)\\
F_2 & = & R(-3c-2d,-b-a) \oplus R(-3c-d,-a-2b) \oplus R(-2c-d,-2a-b).
\end{eqnarray*}
\normalsize
\end{corollary}

\begin{proof}
Because $I_X$ is an almost complete intersection, Corollary
\ref{P1xP1 result}  implies that $I_Z = I_{X}^{(3)} = I_X^3$.
So the bigraded resolution of $I_Z$ can be computed using
Corollary
\ref{cor:res powers} if we know the bigraded resolution
of $I_X$.  But by \cite[Theorem 5.3]{GVBook} the bigraded resolution of an almost complete intersection $I_X$
with  $\alpha_X =
(\underbrace{a,\ldots,a}_c,\underbrace{b,\ldots,b}_d)$
 is
\[0 \rightarrow R(-c-d,-b) \oplus R(-c,-a)
\rightarrow R(-c-d,0) \oplus R(-c,-b) \oplus R(0,-a) \rightarrow I_X
\rightarrow 0.\]
\end{proof}

\begin{remark}
In the previous statement, the set of fat points all have
multiplicity three.  Favacchio and Guardo \cite{FG} have
generalized this result.  In particular, they have shown that if
$Z$ is set of fat points in $\mathbb{P}^1 \times \mathbb{P}^1$
whose support is an almost complete intersection,  one can  weaken the hypothesis that the fat points are all homogeneous of degree three
to construct a set of nonhomogeneous fat points (in a controlled fashion),
 and still prove that $I_Z^{(m)} = I_Z^m$ for
all $m \geq 1.$
\end{remark}

%%%%%%%%%%%%%%%%%%%%%%%%%%%%%%%%%%%%%%%%%%%%%%%%%%%%%%%%%%%%%%%%%%%%%

\section{Application 2: a Question of R\"{o}mer}
\label{sect:Romer}

In this last section we show how one can use Theorem \ref{thm:EN exact} to give
further evidence for a question of R\"omer \cite{R}.  We begin by
defining and introducing the relevant notation.

Let $I$ be a homogeneous ideal of $R = k[x_0,\ldots,x_n]$.
The  graded minimal free resolution of $R/I$ has the form
\[0 \rightarrow F_p \rightarrow F_{p-1}
\rightarrow \cdots \rightarrow F_1 \rightarrow R
 \rightarrow R/I \rightarrow 0 \]
where $F_i = \bigoplus_{j\in\mathbb{Z}}
R(-j)^{\beta_{i,j}(R/I)}$. The number $p =
\operatorname{proj-dim}(R/I)$ is the {\it projective dimension},
and the numbers $\beta_{i,j}(R/I)$
are the $(i,j)$-th {\it graded Betti numbers}
of $R/I$.  The {\it $i$-th Betti number of} $R/I$ is
$\beta_i(R/I) = \sum_{j \in \mathbb{Z}} \beta_{i,j}(R/I)$.

R\"omer \cite{R} initiated an investigation into the relationship
between the $i$-th Betti numbers of $I$ and the shifts that appear
in the graded minimal free resolution.
In particular, R\"omer asked the following question.

\begin{question}\label{romer}
Let $I$ be a homogeneous ideal of $R = k[x_0,\ldots,x_n]$. Does the
following bound hold for all $i=1,\ldots,p$:
\begin{equation*}
\beta_i(R/I) \leq \frac{1}{(i-1)!(p-i)!}\prod_{j\neq i}{M_j}
\end{equation*}
where $M_i := \max\{j ~|~ \beta_{i,j}(R/I) \neq 0\}$?
\end{question}

\noindent
R\"{o}mer showed that Question \ref{romer} is true for all
codimension two Cohen-Macaulay ideals, while Guardo and Van Tuyl
\cite{GVT2} verified the question for any set of double
points in $\P^1 \times \P^1$ whose support is ACM.  We now
use Theorem \ref{thm:EN exact} to extend the family of known
positive answers to Question \ref{romer}.  In the statement
below, recall that $\alpha(I) = \min\{i ~|~ (I)_i \neq 0\}$.

\begin{theorem}\label{powercodim2}
Consider a homogeneous perfect
ideal $I \subset R$ of codimension two
that is also locally a complete intersection.
Fix an integer $m \in \{1,\ldots,n\}$.
% and let
%$\mu(I) = \beta_0(I)$  be the number of minimal generators of $I$.
%If $\mu(I) \leq m\alpha(I)$,
Then
\[
\beta_i(R/I^m) \leq \frac{1}{(i-1)!(m+1-i)!}\prod_{j\neq i}{M_j}
~~\mbox{for all $1\leq i \leq m+1$}
\]
where $M_i := \max\{j ~|~ \beta_{i,j}(R/I) \neq 0\}$.
\end{theorem}

\begin{proof}   If $m = 1$, then this result follows from \cite[Corollary 5.2]{R}.
Let $d = \mu(I)$.  By the Hilbert-Burch Theorem,
the ideal $I$ has a minimal resolution of the form:
\[0\rightarrow R^{d-1}
{\rightarrow} R^d \rightarrow
I\rightarrow 0.\]
Furthermore, the minimal generators of $I$ are given by the $(d-1) \times (d-1)$
minors of the Hilbert-Burch matrix.
Since every entry of this matrix is either $0$ or has degree $\geq 1$, we have
$\alpha(I) \geq d-1$.  In particular,  for all $m \geq 2$, we have
$d \leq m\alpha(I)$.

Because of our  hypotheses on $I$ and $m$, Theorem \ref{thm:EN exact}
implies that the complex of Lemma \ref{lem:EN-complex} is a graded minimal
free resolution of $I^m$, and consequently, proj-dim$(R/I^m) = m+1$
and for all $1 \leq i \leq m+1$,
\begin{eqnarray*}
\beta_i(R/I^m) & = &\dim_k \left(\bigwedge^{i-1} R^{d-1} \otimes {\rm Sym}^{m-i+1}(R^d)\right)
= \binom{d-1}{i-1}\binom{d+m-i}{d-1}.
\end{eqnarray*}

We have $\alpha(I^m) = m\alpha(I)$.  As a result, for each $i=1,\ldots,m+1$,
$\beta_{i,j}(R/I^m) = 0$ for all $j < m\alpha(I)+(i-1)$.
So, because $d \leq m\alpha(I)$,
we have
\[d+(i-1) \leq m\alpha(I) + (i-1) \leq M_i ~~\mbox{for all $i=1,\ldots,m+1$.}\]

Combining these pieces together now gives
\begin{eqnarray*}
\beta_i(R/I^m) &= &\frac{(d-1)(d-2)\cdots (d-i+1)(d-i)!}{(i-1)!(d-i)!}
\frac{(d+m-i)\cdots (d)(d-1)!}{(m+1-i)!(d-1)!} \\
& = & \frac{\left((d-i+1)(d-i+2)\cdots (d-1)\right)\left((d)(d+1)\cdots (d+m-i) \right)}{(i-1)!(m+1-i)!} \\
& \leq & \frac{\left((d)(d+1) \cdots (d-i-2)\right)\left((d+i)(d+i+1) \cdots (d+m)\right)}{(i-1)!(m+1-i)!} \\
& \leq & \frac{M_1M_2 \cdots M_{i-1}M_{i+1}\cdots M_{m+1}}{(i-1)!(m+1-i)!}.
\end{eqnarray*}
This now verifies the inequality.
\end{proof}

We can reprove and extend \cite[Theorem 6.1]{GVT2}
which first proved the case $m=2$.

\begin{corollary}  Let $X \subseteq \P^1 \times \P^1$ be an ACM set of
points.  If $I = I_X$, then   Question \ref{romer}  has an affirmative answer for $I^m$ with
$m=1,2$ and $3$.
\end{corollary}

%%%%%%%%%%%%%%%%%%%%%%%%%%%%%%%%%%%%%%%%%%%%%%%%%%%%%%%%%%%%%%%%%%%%

%%%%%%%%%%%%%%%%%%%%%%%%%%%%%%%%%%%%%%%%%%%%%%%%%%%%%%%%%%%%%%

\end{document}